\documentclass[reqno]{amsart}

\usepackage{dsfont,mathrsfs}
\usepackage{amssymb,amscd,amsfonts,amsbsy,amsmath,amsthm,amsrefs}

\begin{document}
\title[Maximal regularity and PDEs]{Maximal regularity as a tool for partial differential equations}

\author[Sylvie Monniaux\hfil \hfilneg]{Sylvie Monniaux}  % in alphabetical order

\address{Sylvie Monniaux \newline
I2M - CMI - Technop\^ole de Ch\^ateau-Gombert - 39 rue Fr\'ed\'eric Joliot-Curie - 13453 Marseille Cedex 13}
\email{sylvie.monniaux@univ-amu.fr}

\subjclass[2000]{47D06, 34G10, 35B45, 35Q30} 
\keywords{Maximal regularity, partial differential equations, evolution equations, 
Navier-Stokes equations, existence of solutions, uniqueness of solutions}

\begin{abstract}
In the last decades, a lot of progress has been made on the subject of maximal regularity. The property of maximal $L^p$ regularity is an 
a priori estimate and reads as follows:

For A the negative generator of an analytic semigroup on a Banach space $X$, for $1<p<\infty$, for $0<T<=\infty$, 
does there exist $C_p>0$ a constant such that for all $f\in L^p(0,T;X)$, 
there exists a unique solution $u\in L^p(0,T;D(A))\cap W^{1,p}(0,T;X)$ of the equation $\partial_t(u)+Au=f$, $u(0)=0$ with the estimate
\[
\|\partial_t(u)\|_{L^p(0,T;X)}+\|Au\|_{L^p(0,T;X)}\le C_p \|f\|_{L^p(0,T;X)} ?
\]
It started with a paper by De Simon in 1964 in which the author proved maximal $L^p$ regularity for negative generators of 
analytic semigroups in Hilbert spaces. The next big step has been made by Dore and Venni in their 1987 paper on operators 
admitting bounded imaginary powers. The final result on maximal regularity was done by Weis in his 2001 paper where he 
gave a characterisation of negative generators of analytic semigroups in Banach spaces which have the maximal regularity 
property. These results are all about linear theory of unbounded operators in Banach spaces.

I will then show how to use this property to find solutions or to prove uniqueness of solutions of semi linear partial differential equations 
of parabolic type such as the non linear heat equation and the Navier-Stokes system.
\end{abstract}

\maketitle \numberwithin{equation}{section}
\newtheorem{theorem}{Theorem}[section]
\newtheorem{corollary}[theorem]{Corollary}
\newtheorem{proposition}[theorem]{Proposition}
\newtheorem{lemma}[theorem]{Lemma}

\theoremstyle{definition}
\newtheorem{remark}[theorem]{Remark}
\newtheorem{problem}[theorem]{Problem}
\newtheorem{example}[theorem]{Example}
\newtheorem{definition}[theorem]{Definition}
\allowdisplaybreaks

\section{Introduction}

The property of maximal $L^p$ regularity ($1<p<\infty$) can be formulated as follows:

let $A$ be an (unbounded) operator with dense domain ${\mathsf D}(A)$ defined on a Banach space $X$, 
let $0<T\le \infty$, for $f\in L^p(0,T;X)$, does there exist 
$u\in L^p(0,T;{\mathsf D}(A))$ such that $\partial_tu+Au=f$ and $u(0)=0$? And in that case, the closed graph theorem
implies that there is a constant $C>0$ such that
\begin{equation}
\label{eq:maxreg}
\|u\|_{L^p(0,T;X)}+\|\partial_t u\|_{L^p(0,T;X)}+\|Au\|_{L^p(0,T;X)}\le C \|f\|_{L^p(0,T;X)}.
\end{equation}
It is clear that if $A$ has the maximal $L^p$ regularity property for $T=\infty$, then this is also the case for $T<\infty$. 

In a first part, we will study this property of maximal $L^p$ regularity: what does this imply on the operator $A$? What are the operators
with this property? The non zero initial condition will also be investigated.

In a second part, we will focus on how to apply this property to study non linear equations, and more precisely semi linear problems.
Examples will range from the non linear heat equation to the Navier-Stokes system: existence of solutions, uniqueness of mild solutions...

\section{Maximal regularity}

\subsection{Analytic semigroup}

The maximal $L^p$ regularity property for operators defined on a Hilbert space has been studied by De Simon \cite{DeS64} in 1964. 
Let us start with easy properties that an operator with the maximal $L^p$ property enjoys.

\begin{proposition}
\label{prop:analyticsg}
Let $p\in(1,\infty)$ and $T = \infty$. Assume that $A$ has the maximal $L^p$ regularity property. Then $-A$ generates an analytic semigroup.
\end{proposition}

\begin{proof}
Let $z\in{\mathds{C}}$ with $\Re e(z)>0$ and define the function $f_z\in L^p(0,\infty)$ by 
\[
f_z(t)=e^{zt} \mbox{ if }0\le t\le \frac{1}{\Re e(z)} \mbox{ and }
f_z(t)=0 \mbox{ if }t>\frac{1}{\Re e(z)}.
\] 

Let $x\in X$: denote by $u_z$ the solution of $\partial_t u+A u=f_z x$, $u_z(0)=0$ and let
\[
R_zx:=\Re e(z)\int_0^\infty e^{-zt}u_z(t)\,{\rm d}t.
\] 
One can prove that, using \eqref{eq:maxreg}, $\|R_zx\|_X\le C_p \|x\|_X$.
An easy integration by parts gives also that 
\[
R_zx=\frac{\Re e(z)}{z}\int_0^\infty e^{-zt}\partial_t u_z(t)\,{\rm d}t
\] 
and that $\|R_zx\|_X\le C_p'\frac{1}{|z|}\|x\|_X$. Putting these two estimates together gives
\[
\|R_zx\|_X\le M_p\,\frac{1}{1+|z|}\|x\|_X,
\]
where $M_p$ is a constant depending only on $p$ and on the constant $C$ in \eqref{eq:maxreg}.

Now, for $x\in{\mathsf D}(A)$, it is easy to see that $R_z(zx+Ax)=x$, so that since ${\mathsf D}(A)$ is dense in $X$
$R_z$ is the resolvent of $-A$ in $z$.
\end{proof}

\begin{remark}
Since $-A$ generates an analytic semigroup $(e^{-tA})_{t\ge 0}$, the solution $u$ of $\partial_t u+Au=f$, $u(0)=0$ is
given by
\begin{equation}
\label{eq:Duhamel}
u(t)=\int_0^t e^{-(t-s)A}f(s)\,{\rm d}s, \quad t\ge0.
\end{equation}

\end{remark}

The next result concerns the independence of the maximal $L^p$ regularity property with respect to $p\in(1,\infty)$.

\begin{proposition}
\label{prop:indep/p}
Let $p_0\in(1,\infty)$. Assume that $A$ is the negative generator of an analytic semigroup with the maximal $L^{p_0}$ regularity property.
Then $A$ has the maximal $L^p$ regularity property for all $p\in(1,\infty)$.
\end{proposition}

The proof of this proposition relies on a more general result by A. Benedek, A.P. Calder\'on, R. Panzone \cite[Theorem~2]{BCP62}.

\begin{theorem}
Let $X$ be a Banach space. Let $p\in(1,\infty)$ and denote by $k:{\mathds{R}}\to {\mathscr{L}}(X)$ a measurable function 
$k\in L^1_{\rm loc}({\mathds{R}}\setminus\{0\},{\mathscr{L}}(X))$. Let $S\in{\mathscr{L}}(L^p({\mathds{R}},X))$ be the 
convolution operator with $k$, i.e.,
\[
Sf(t)=\int_{\mathds{R}}k(t-s)f(s)\,{\rm d}s, \quad \forall \,f\in L_{\rm c}^\infty({\mathds{R}},X), t\notin{\rm supp}(f).
\]
Assume that there is a constant $c>0$ such that 
\[
\int_{|t|>2|s|}\|k(t-s)-k(t)\|_{{\mathscr{L}}(X)}\,{\rm d}t\le c,\quad \forall\,s\in{\mathds{R}}.
\]
Then $S\in {\mathscr{L}}(L^q({\mathds{R}},X))$ for all $q\in(1,\infty)$.
\end{theorem}

\begin{proof}
The proof relies on a Calder\'on-Zygmund decomposition of $f$ to prove that $S$ is of weak type (1,1). Then by interpolation 
(Marcinkiewicz theorem), one deduces boundedness of $S$ on $L^q({\mathds{R}},X)$ for all $q\in (1,p_0]$. To prove boundedness
for $q>p_0$, we argue by duality.
\end{proof}

\begin{proof}[Proof of Proposition~\ref{prop:indep/p}]
Relying on \eqref{eq:Duhamel}, we apply previous theorem with 
\[
k(t)=A e^{-tA} \mbox{ for }t>0 \mbox{ and }k(t)=0 \mbox{ if }t\le 0.
\]
It suffices to show that this kernel $k$ satisfies the assumptions of the theorem.
\end{proof}

We finish this first insight into maximal $L^p$ regularity by giving a reverse statement of Proposition~\ref{prop:analyticsg}
in the case $X$ is a Hilbert space. This result is due to De Simon \cite{DeS64}.

\begin{theorem}
\label{thm:maxregHilbert}
Let $A$ be the negative generator of a bounded analytic semigroup on a Hilbert space $X$. Then $A$ has the
maximal $L^p$ regularity.
\end{theorem}

\begin{proof}
Using \eqref{eq:Duhamel} to represent the solution of $\partial_t u+Au=f$, $u(0)=0$ with $f\in L^2(0,\infty;X)$ and
taking the Fourier-in-time transform ${\mathscr{F}}$ of $Au$ shows that ${\mathscr{F}}(Au)\in L^2(0,\infty;X)$ (since
$-A$ generates a bounded analytic semigroup) and then $\|Au\|_{L^2(0,\infty;X)}\lesssim \|f\|_{L^2(0,\infty;X)}$.
Note that
\[
{\mathscr{F}}(Au)(\tau)=A(i\tau+A)^{-1}{\mathscr{F}}(f)(\tau), \quad \tau\in{\mathds{R}}.
\]
We conclude by applying Proposition~\ref{prop:indep/p}.
\end{proof}

It has been a long open problem whether this last theorem was true in general Banach spaces or not. A result by
T. Coulhon and D. Lamberton \cite{CL86} shows that it is a necessary condition that the Banach space $X$ has the $UMD$ 
property, i.e., the Hilbert transform
\[
{\mathscr{H}}:L^2({\mathds{R}};X)\to L^2({\mathds{R}};X);\quad 
{\mathscr{H}}(f)(t):=\frac{1}{\pi} {\rm p.v.}\int_{\mathds{R}}\frac{f(s)}{t-s}\,{\rm d}s,\ \ \ t\in{\mathds{R}}
\]
is a bounded operator. Examples of $UMD$ spaces are: Hilbert spaces, $L^p(\Omega,X)$ ($\Omega \subset{\mathds{R}}^d$
for $p\in(1,\infty)$ if $X$ is a $UMD$ space. A $UMD$ Banach space is necessarily reflexive. 

\subsection{Maximal regularity in Banach spaces - weighted maximal regularity}

We focus now on the maximal $L^p$ regularity property for negative generators of analytic semigroups in $UMD$ Banach space.

D. Lamberton \cite{La87} proved that the negative generator of an analytic semigroup of contractions on $X=L^2(\Omega,\mu)$ (where
$(\Omega,\mu)$ is a measure space) satisfying $\|e^{-tA}f\|_q\lesssim\|f\|_q$ for all $f\in L^2(\Omega,\mu)\cap L^q(\Omega,\mu)$ 
for all $q\in (1,\infty)$ has the maximal $L^p$ regularity property.

\begin{example}
The negative Laplacian with Dirichlet, Neumann or Robin boundary conditions (on a domain for which the divergence theorem 
for $L^1$ functions applies; this is the case for instance if it is a domain with a Lipschitz boundary) has the maximal $L^p$ regularity property.
\end{example}

A positive result was given by G. Dore and A. Venni the same year in \cite{DV87} where they proved that if $A$ has bounded 
imaginary powers with angle strictly less than $\frac{\pi}{2}$, then it has the maximal $L^p$ regularity property. 

Later, M. Hieber and J. Pr\"u{\ss} \cite{HP97} and 
T. Coulhon and X.T. Duong \cite{CD00} proved that if the semigroup generated by $-A$ has a kernel with gaussian estimates, then
$A$ has the maximal $L^p$ regularity property. This result was generalised by P.C. Kunstmann \cite{Ku08} for negative generators 
of semigroups with integrated gaussian estimates.

A characterisation of operators with the maximal $L^p$ regularity property was given by L. Weis \cite{We01} using the concept of 
${\mathcal{R}}$ boundedness.

\begin{definition}
\label{def:Rbdd}
A family $\tau$ of bounded linear operators from $X$ to $Y$ is called ${\mathcal{R}}$ bounded if there exists a constant $C>0$
such that for all $n\ge 1$, for all $T_1,...,T_n\in \tau$ and all $x_1,...,x_n\in X$,
\[
\int_0^1\bigl\|\sum_{j=1}^nr_j(s)T_jx_j\bigr\|_Y\,{\rm d}s\le C \int_0^1\bigl\|\sum_{j=1}^nr_j(s)x_j\bigr\|_X\,{\rm d}s
\]
where $(R_j)_{j=1,...,n}$ is a sequence of independent $\{-1,1\}$ valued random variables on $[0,1]$ ; for example, the Rademacher
functions $r_j(t)={\rm sgn} \bigl(\sin (2^j\pi t)\bigr)$. The ${\mathcal{R}}$ bound of $\tau$ denoted by ${\mathcal{R}}(\tau)$ is then the 
smallest $C$ for which the above estimate holds.
\end{definition}

\begin{remark}
If $X$ and $Y$ are Hilbert spaces, then a family $\tau$ is ${\mathcal{R}}$ bounded if, and only if, it is bounded.
\end{remark}

We are now in the position to state the result which characterises operators with the maximal $L^p$ regularity property 
due to L. Weis \cite{We01}.

\begin{theorem}
\label{thm:maxregRbdd}
Let $A$ be the negative generator of an analytic semigroup on a $UMD$ Banach space. Then $A$ has the maximal $L^p$ regularity
property if, and only if, the family of resolvents $\{i\sigma(i\sigma{\rm I}+A)^{-1}, \sigma\in{\mathds{R}}\}$ is ${\mathcal{R}}$ bounded.
\end{theorem}

We finish this section with the notion of weighted maximal $L^p$ regularity. The following result is due to J. Pr\"u{\ss} and G. Simonett
\cite[Theorem~2.4]{PS04}.

\begin{theorem}
\label{thm:weightedmaxreg}
Assume that $A$ is an operator having the maximal $L^p$ regularity property on a Banach space $X$ with the UMD property. 
Then $A$ has the weighted maximal $L^p$ regularity property, i.e., for all $\mu\in(\frac{1}{p},1]$, for all $f\in L^p_\mu(0,\infty;X)$,
there exists a unique solution $u\in L^p_\mu(0,\infty;{\sf D}(A))$ with $\partial_t u\in L^p_\mu(0,\infty;X)$ of $\partial_t u+Au=f$
with $u(0)=0$, where
\[
L^p_\mu(0,\infty;X)=\bigl\{u:(0,\infty)\to X ; t\mapsto t^{1-\mu}u(t)\in L^p(0,\infty;X)\bigr\}
\]
equipped with the norm $\|u\|_{L^p_\mu(0,\infty;X)}=\|t\mapsto t^{1-\mu}u(t)\|_{L^p(0,\infty;X)}$.
\end{theorem}

\begin{remark}
The weighted maximal $L^p$ regularity property for $\mu=1$ is the usual maximal $L^p$ regularity property.
\end{remark}

\section{Semi-linear equations in critical spaces}

In this section, we describe two semi-linear pdes of the form
\[
\partial_tu+Au=f(u), \quad u(0)=u_0.
\]
We will give existence and uniqueness results using the maximal
$L^p$ regularity property.

\subsection{Mild solutions}

First, let us define what we mean with ``critical space" for a pde of the previous form. Assume that $A$ is an operator on a space
depending of a space variable $x$ homogeneous in the sense that $Au_\lambda (x) = \lambda^\alpha Au(x)$ for all $\lambda>0$, 
$x\in {\mathds{R}}^n$ where $u_\lambda(x)=u(\lambda x)$. Assume moreover that the function $f$ is also 
homogeneous in the sense that $f(\lambda w)=\lambda^\gamma f(w)$ and $f(v_\lambda(t,x))=\lambda^\beta f(v(\lambda^\alpha t,\lambda x))$ 
where $v_\lambda(t,x)=v(\lambda^\alpha t, \lambda x)$ for $\lambda>0$, $t>0$ and $x\in\Omega\subset{\mathds{R}}^n$. We assume that
if $\alpha\neq\beta$, then $\gamma\neq1$. In that case, if $u$ is a solution of $\partial_t u+Au=f(u)$, then 
$(t,x)\mapsto \lambda^{\frac{\alpha-\beta}{\gamma-1}}u(\lambda^\alpha t, \lambda x)$ is also a solution of the same equation
for all $\lambda>0$.

\begin{definition}
\label{def:criticalspace}
In this situation, we call $Y$ a critical space for $\partial_tu+Au=f(u)$ a time-space variable space for which
\[
\bigl\|(t,x)\mapsto \lambda^{\frac{\alpha-\beta}{\gamma-1}}u(\lambda^\alpha t, \lambda x)\bigr\|_Y=\|u\|_Y.
\]
\end{definition}

\begin{example}
If $Y$ is of the form $Y=L^p_t(L^q_x)$ for $1<p,q<\infty$, $Y$ is a critical space if, and only if, 
$\frac{\alpha-\beta}{\gamma-1}=\frac{n}{q}+\frac{\alpha}{p}$.
\end{example}

\begin{definition}
\label{def:mildsol}
Assume that $A$ is the negative generator of an analytic semigroup on a Banach space $X$.
A mild solution on $[0,T]$  ($0<T\le+\infty$) of the equation $\partial_t u+Au=f(u)$ with initial value $u_0\in X$ is a 
continuous function in $t\in [0,T)$ with values in $X$ satisfying the Duhamel formula:
\[
u(t)=e^{-tA}u_0+\int_0^t e^{-(t-s)A}f(u(s))\,{\rm d}s, \quad 0\le t< T.
\]
In other words, $u$ is a fixed point of the map $v\mapsto a+F(v)$ where $a(t)=e^{-tA}u_0$, $t\ge 0$ and
$F(v)(t)=\int_0^t e^{-(t-s)A}f(v(s))\,{\rm d}s$, $0\le t < T$. If $T=+\infty$, $u$ is called a global mild solution.
\end{definition}

The method to prove the existence of global mild solutions for small initial data or local mild solutions if the initial
condition has no size restriction relies on the following fixed point theorem.

\begin{theorem}[Fixed point theorem]
\label{thm:fixedpoint}
Let $Y$ be a Banach space, $a\in Y$ and $F:Y\to Y$ with the property that $F(0)=0$ and there exists $\epsilon >0$ 
and $M>0$ such that
\begin{equation}
\label{est:F}
\|F(u)-F(v)\|_Y\le M \|u-v\|_Y\bigl(\|u\|_Y^\epsilon+\|v\|_Y^\epsilon\bigr),\quad \forall\,u,v\in Y.
\end{equation}
Then for $\delta <\frac{1}{2(2M)^{\frac{1}{\epsilon}}}$, for all $a\in \overline{B}_Y(0,\delta)$, the ball in $Y$ of center $0$
and radius $\delta$, the map $\Phi:u\mapsto a+F(u)$ admits a fixed point, unique in $\overline{B}_Y(0,2\delta)$.
\end{theorem}

\begin{proof}
The proof of the theorem follows the lines of \cite[Theorem~15.1]{LR02}, Picard contraction priniciple. 

The choice of $\delta$ implies that $(2\delta)^\epsilon<\frac{1}{2M}$.

First, we prove that if $a\in \overline{B}_Y(0,\delta)$, then $\overline{B}_Y(0,2\delta)$ is stable under the map $\Phi$.  
The estimate on $F$ and the fact that $F(0)=0$ give for $u\in \overline{B}_Y(0,2\delta)$ the estimate
$\|F(u)\|_Y\le M\|u\|^{1+\epsilon}\le M (2\delta)^{1+\epsilon}\le \delta$. With $\|a\|_Y\le \delta$, we
obtain $\|\Phi(u)\|_Y\le 2\delta$.

Next, $\Phi:\overline{B}_Y(0,2\delta)\to \overline{B}_Y(0,2\delta)$ is a contraction. We have that
\[
\|\Phi(u)-\Phi(v)\|_Y\le M(\|u\|_Y^\epsilon+\|v\|_Y^\epsilon)\|u-v\|_Y\le 2M(2\delta)^\epsilon\|u-v\|_Y
\] 
with $2M(2\delta)^\epsilon<1$ as established in the previous step.

We conclude by Picard contraction principle that $\Phi$ has a unique fixed point in $\overline{B}_Y(0,2\delta)$.
\end{proof}

\subsection{The non linear heat equation: existence and uniqueness of solutions}

The non linear heat equation in ${\mathds{R}}^n$ has the form
\begin{equation}
\label{NLHE1}
\tag{NLHE1}
\left\{
\begin{array}{rcl}
\partial_tu-\Delta u&=& |u|^{\nu-1}u\quad \mbox{ in }(0,\infty)\times{\mathds{R}}^n\\
u(0)&=&0 \quad \mbox{ in }{\mathds{R}}^n
\end{array}
\right.
\end{equation}
or
\begin{equation}
\label{NLHE2}
\tag{NLHE2}
\left\{
\begin{array}{rcl}
\partial_tu-\Delta u&=& |u|^\nu \quad \mbox{ in }(0,\infty)\times{\mathds{R}}^n\\
u(0)&=&0 \quad \mbox{ in }{\mathds{R}}^n
\end{array}
\right.
\end{equation}
for $1<\nu<\infty$. We propose here to give a proof of existence and uniqueness of solutions of \eqref{NLHE1} or \eqref{NLHE2}
as in \cite[Theorem~1]{W80} using the property of maximal regularity (see also \cite{W81}. With the notations used to define a critical space for 
these equations, we have in our case $\alpha=2$, $\beta=0$ and $\gamma=\nu$ so that a critical space for
\eqref{NLHE1} or \eqref{NLHE2} is of the form $L^p_t(L^q_x)$ with $\frac{n}{q}+\frac{2}{p}=\frac{2}{\nu-1}$. In the special case
$\nu=2$ (the quadratic non linear heat equation), this condition becomes $\frac{n}{q}+\frac{2}{p}=2$.

\begin{theorem}[Existence]
\label{thm:existNLHE}
Let $\nu\in(1,\infty)$.
Assume that $\nu<p,q<\infty$ satisfy $\frac{n}{q}+\frac{2}{p}=\frac{2}{\nu-1}$.
There exists $\eta>0$ such that for all $u_0\in \dot B^{-2/p}_{p,q}$ with $\|u_0\|_{\dot B^{-2/p}_{p,q}}\le \eta$,
the equations \eqref{NLHE1} and \eqref{NLHE2} have a mild solution in $L^p_t(L^q_x)$.
\end{theorem}

\begin{proof}
The treatments of \eqref{NLHE1} and \eqref{NLHE2} are the same, so we will focus only on solutions of \eqref{NLHE1}.
We start by rewriting \eqref{NLHE1} in the integral form as in Definition~\ref{def:mildsol}: $u\in Y=L^p_t(L^q_x)$ is a mild
solution of \eqref{NLHE1} if
\begin{equation}
\label{eq:mildNLHE}
u(t)= e^{t\Delta}u_0+\int_0^t e^{(t-s)\Delta}|u(s)|^{\nu-1}u(s)\,{\rm d}s, \quad t\in(0,\infty),
\end{equation}
and define $a:t\mapsto e^{t\Delta}u_0$ and
\begin{equation}
\label{def:F}
F(u)(t)=\int_0^t e^{(t-s)\Delta}|u(s)|^{\nu-1}u(s)\,{\rm d}s, \quad t\in(0,\infty).
\end{equation}
In that case, $u$ is a mild solution in $Y$ if $u=a+F(u)$.
{\sl Claim~1: $a$ belongs to $Y$.}

\noindent
By definition of the Besov space $\dot B^{-2/p}_{p,q}$, $a\in Y$ and $\|a\|_Y\le c\eta$ where $c$ is a constant.

\medskip

\noindent
{\sl Claim~2: $F$ maps $Y$ to $Y$.}

\noindent
This comes from the maximal regularity property satisfied by the heat semigroup and
the mixed derivative result:
\begin{equation}
\label{eq:mixedder}
W^{1,p}_t(L^q_x)\cap L^p_t(W^{2,q}_x)\hookrightarrow W^{s,p}_t(W^{2(1-s),q}_x),\quad s\in(0,1).
\end{equation}
For $u\in L^p_t(L^q_x)$, we have that $|u|^{\nu-1}u\in L^{p/\nu}_t(L^{q/\nu}_x)$ and the mixed derivative result
for $p/\nu$ instead of $p$, $q/\nu$ instead of $q$ and $s=\frac{\nu-1}{p}$, using the condition linking $p$ and $q$
and the Sobolev embeddings 
\begin{equation}
\label{eq:sobemb}
W^{\frac{\nu-1}{p},\frac{p}{\nu}}_t\hookrightarrow L^p_t \mbox{ in dimension 1}\quad\mbox{ and }\quad 
W^{\frac{n(\nu-1)}{q},\frac{q}{\nu}}_x\hookrightarrow L^q_x \mbox{ in dimension }n 
\end{equation}
yield the claim.

\medskip

\noindent
{\sl Claim~3: $F$ satisfies \eqref{est:F} with $\epsilon=\nu-1$.}

\noindent
It is immediate that $F(0)=0$. Next,
for all $x,y\in{\mathds{R}}$, we have 
\begin{equation}
\label{eq:estxnu}
\bigl||x|^{\nu-1}x-|y|^{\nu-1}y\bigr|\le \nu (|x|^{\nu-1}+|y|^{\nu-1})|x-y|.
\end{equation}
 Therefore using the same tools as in Claim~1, we obtain
\[
\|F(u)-F(v)\|_Y\le M \|u-v\|_Y (\|u\|_Y^{\nu-1}+\|v\|_Y^{\nu-1})
\]
with $M=\nu C_{p/\nu,q/\nu} M_{\frac{\nu-1}{p},p/\nu,q/nu} \alpha_{\nu,p}\beta_{\nu,q}$, 
where $C_{p,q}$ is the constant appearing in \eqref{eq:maxreg} for $T=\infty$, $X=L^q_x$ and $p$,
$M_{s,p,q}$ is the embedding constant in the mixed derivative result \eqref{eq:mixedder}, $\alpha_{\nu,p}$ and 
$\beta_{\nu,q}$ are the constants in the Sobolev embeddings \eqref{eq:sobemb}.

\medskip

\noindent
To end the proof, it suffices to apply Theorem~\ref{thm:fixedpoint} for $\delta = c\eta$ ($\eta$ small enough so that
the condition on $\delta$ is satisfied). 
\end{proof}

The next result is a uniqueness result for solutions of \eqref{NLHE1} or \eqref{NLHE2} in ${\mathscr{C}}([0,T];L^q_x)$.
The limiting case $p=\infty$ in the previous result gives $q=\frac{n(\nu-1)}{2}$: the condition $\frac{q}{\nu}>1$ implies then
that $n>\frac{2\nu}{\nu-1}$. Remark that in the case of the quadratic non linear heat equation ($\nu=2$), this imposes that 
we work in dimension $n\ge 5$.

\begin{theorem}[Uniqueness]
\label{thm:uniquenessNLHE}
Let $\nu>1$, let $n>\frac{2\nu}{\nu-1}$, let $q=\frac{n(\nu-1)}{2}$. Let $T>0$. Then for all $u_0\in L^q_x$ there exists at most one mild solution
in ${\mathscr{C}}([0,T];L^q_x)$ of \eqref{NLHE1} or \eqref{NLHE2}.
\end{theorem}

\begin{proof}
The strategy of the proof is to assume that there are two mild solutions $u$ and $v$ in ${\mathscr{C}}([0,T];L^q_x)$
for the same initial value $u_0$ so that $u-v=F(u)-F(v)$ for $F$ defined in \eqref{def:F}. Then we show that $u$ and $v$ 
necessarily coincide in $L^p(0,\tau;L^q_x)$ for $\tau>0$ and one $p>1$: $\|u-v\|_{L^p(0,\tau;L^q_x)}=0$. 
Iterating this procedure ultimately shows that $u=v$ on $[0,T]$.

\medskip

\noindent
{\bf Step 1:}
Let $p>1$. Using \eqref{eq:estxnu}, we have the following estimate for every $0<\tau\le T$
\[
\|F(u)-F(v)\|_{L^p(0,\tau;L^q_x)}\le M \|u-v\|_{L^p(0,\tau;L^q_x)}\bigl(\|u\|_{L^\infty(0,\tau;L^q_x)}^{\nu-1}+\|v\|_{L^\infty(0,\tau;L^q_x)}^{\nu-1}\bigr)
\]
with $M$ independent of $\tau$. This follows from the same arguments as in the proof of Claim~3 of Theorem~\ref{thm:existNLHE}. 

\medskip

\noindent
{\bf Step 2:}
Let $\epsilon >0$.
We can approximate $u_0\in L^q_x$ by a $u_{0,\epsilon}\in {\mathscr{C}}_c^\infty({\mathds{R}}^n)$ such that
\[
\|u_{0,\epsilon}-u_0\|_{L^q_x}\le \epsilon.
\]
We denote by $a_\epsilon$ the function defined on $(0,\infty)$ by $a_\epsilon(t)= e^{t\Delta}u_{0,\epsilon}$ and we have for all $\epsilon>0$
\begin{align}
\label{eq:u-aepsilon}
&\|u\|^{\nu-1}_{L^\infty(0,\tau;L^q_x)}-\|a_\epsilon\|^{\nu-1}_{L^\infty(0,\tau;L^q_x)}
\xrightarrow[\tau\to 0]{} \|u_0\|_{L^q_x}^{\nu-1}-\|u_{0,\epsilon}\|_{L^q_x}^{\nu-1} \nonumber\\ 
\mbox{and }\quad&\|v\|^{\nu-1}_{L^\infty(0,\tau;L^q_x)}-\|a_\epsilon\|^{\nu-1}_{L^\infty(0,\tau;L^q_x)}
\xrightarrow[\tau\to 0]{} \|u_0\|_{L^q_x}^{\nu-1}-\|u_{0,\epsilon}\|_{L^q_x}^{\nu-1}.
\end{align}

\medskip

\noindent
{\bf Step 3:}
Since $s\mapsto |a_\epsilon(s)|^{\nu-1} (u(s)-v(s))$ belongs to 
$L^p(0,\tau;L^{\frac{nq}{n+q}}_x)$ whenever $u,v\in L^p(0,\tau;L^q_x)$ and $a_\epsilon \in L^\infty(0,\tau;L^{\frac{n}{\nu-1}}_x)$,
we have 
\begin{align*}
&\bigl\|t\mapsto \int_0^t e^{(t-s)\Delta}|a_\epsilon(s)|^{\nu-1} (u(s)-v(s))\,{\rm d}s\bigr\|_{L^p(0,\tau;L^q_x)}\\
\lesssim &\sqrt{\tau} \|u-v\|_{L^p(0,\tau, L^q_x)} \|a_\epsilon\|_{L^\infty(0,\tau;L^{\frac{n}{\nu-1}}_x)}^{\nu-1}
\end{align*}
using the fact that 
\[
\|e^{r\Delta}f\|_{L^q_x}=\|(-\Delta)^{\frac{1}{2}}e^{r\Delta}(-\Delta)^{-\frac{1}{2}}f\|_{L^q_x}
\lesssim \tfrac{1}{\sqrt{r}}\|(-\Delta)^{-\frac{1}{2}}f\|_{L^q_x}\lesssim \tfrac{1}{\sqrt{r}}\|f\|_{L^{\frac{nq}{n+q}}_x}
\] for all $r>0$
and all $f\in L^{\frac{nq}{n+q}}_x$, by the Sobolev embedding $W^{1,\frac{nq}{n+q}}_x\hookrightarrow L^q_x$ in dimension $n$.

\medskip

\noindent
{\bf Step 4:}
For every $\epsilon>0$, we estimate $F(u)-F(v)$ as
\begin{align*}
\bigl|F(u)(t)-F(v)(t)\bigr| \le&
\int_0^t e^{(t-s)\Delta} \bigl|\bigl((|u(s)|^{\nu-1}-|a_\epsilon(s)|^{\nu-1})\\
&+(|v(s)|^{\nu-1}-|a_\epsilon(s)|^{\nu-1})+2|a_\epsilon(s)|^{\nu-1}\bigr)(u(s)-v(s))\bigr|\,{\rm d}s.
\end{align*}
Therefore, applying Step 2 and Step 3, there exists a constant $C>0$ such that
\begin{align*}
\|F(u)-F(v)\|_{L^p(0,\tau;L^q_x)} \le C &\|u-v\|_{L^p(0,\tau;L^q_x)} 
\Bigl(\bigl|\|u\|^{\nu-1}_{L^\infty(0,\tau;L^q_x)}-\|a_\epsilon\|^{\nu-1}_{L^\infty(0,\tau;L^q_x)}\bigr| \\
&+|\|v\|^{\nu-1}_{L^\infty(0,\tau;L^q_x)}-\|a_\epsilon\|^{\nu-1}_{L^\infty(0,\tau;L^q_x)}\bigr|\\
&+\sqrt{\tau}\|a_\epsilon\|^{\nu-1}_{L^\infty(0,\tau;L^{\frac{n}{\nu-1}}_x}\Bigr)
\end{align*}
Since $u-v=F(u)-F(v)$ for $u$ and $v$ mild solutions of \eqref{NLHE1}
We conclude by choosing $\epsilon>0$ small enough such that 
$\bigl|\|u_0\|_{L^q_x}^{\nu-1}-\|u_{0,\epsilon}\|_{L^q_x}^{\nu-1}\bigr|\le \frac{1}{8C}$ and then $\tau>0$ small enough such that
\[
\bigl|\|u\|^{\nu-1}_{L^\infty(0,\tau;L^q_x)}-\|a_\epsilon\|^{\nu-1}_{L^\infty(0,\tau;L^q_x)}\bigr|\le \frac{1}{4C},
\]
\[
\bigl|\|v\|^{\nu-1}_{L^\infty(0,\tau;L^q_x)}-\|a_\epsilon\|^{\nu-1}_{L^\infty(0,\tau;L^q_x)}\bigr|\le \frac{1}{4C},
\]
and 
\[
\sqrt{\tau}\|a_\epsilon\|^{\nu-1}_{L^\infty(0,\tau;L^{\frac{n}{\nu-1}}_x}\le \frac{1}{4C}.
\]
We then obtain that $\|u-v\|_{L^p(0,\tau;L^q_x)} \le\frac{3}{4}\|u-v\|_{L^p(0,\tau;L^q_x)}$ which shows that 
$u=v$ on $[0,\tau)$. 

\medskip

\noindent

{\bf Step 5:} By continuity in time of $u$ and $v$, we have equality on the closed segment $[0,\tau]$.
If $\tau<T$, we can repeat this argument starting from $\tau$ instead of $0$ until we reach $T$.
\end{proof}

\subsection{Navier-Stokes equations}

The same arguments can be applied to Navier-Stokes equations:
\begin{equation}
\label{eq:NS}
\tag{NS}
\left\{
\begin{array}{rclcl}
\partial_tu-\Delta u+\nabla\pi+\nabla\cdot (u\otimes u)&=&0&\mbox{in}&(0,\infty)\times{\mathds{R}}^n\\
{\rm div}\,u&=&0&\mbox{in}&(0,\infty)\times{\mathds{R}}^n\\
u(0)&=&u_0&\mbox{in}&{\mathds{R}}^n,
\end{array}
\right.
\end{equation}
where $u\otimes v$ denotes the matrix $(u_iv_j)_{1\le i,j\le n}$ and $\nabla\cdot M$, the divergence of the matrix 
$M=(m_{i,j})_{1\le i,j\le n}$, is the vector obtained by taking the divergence of the lines:
$\nabla\cdot M= \bigl(\sum_{i=1}^n\partial_im_{i,j}\bigr)_{j=1,...,n}$.
It is easy to see that the invariance for this system is given by $u_{\lambda}(t,x)=\lambda u(\lambda^2t, \lambda x)$ 
and $\pi_{\lambda}(t,x)=\lambda^2\pi((\lambda^2t, \lambda x)$ so that a critical space of the form $L^p_t(L^q_x)$ ($1<p,q<\infty$) for
\eqref{eq:NS} must satisfy the relation $\frac{2}{p}+{n}{q}=1$. A mild solution of \eqref{eq:NS} is given by the formula
\begin{equation}
\label{def:mildsolNS}
u(t)=e^{t\Delta}u_0-\int_0^te^{(t-s)\Delta}{\mathbb{P}}\bigl(\nabla\cdot(u(s)\otimes u(s))\bigr)\,{\rm d}s, \quad t\ge0,
\end{equation}
where ${\mathbb{P}}$ is the Helmholtz projection onto divergence-free vector fields (its Fourier symbol is given by
$\bigl(\delta_{i,j}-\frac{\xi_i\xi_j}{|\xi|^2}\bigr)_{1\le i,j\le n}$). Existence results of mild solutions of \eqref{eq:NS} go back 
to Fujita and Kato \cite{FK64} and Kato \cite{Ka84}.
The same methods as in the case of the non linear heat equation 
apply. We obtain the following existence and uniqueness results.

\begin{theorem}[Existence]
Assume that $1<p,q<\infty$ satisfy $\frac{n}{q}+\frac{2}{p}=1$.
There exists $\eta>0$ such that for all $u_0\in \dot B^{-2/p}_{p,q}$ with ${\rm div}\,u_0=0$ and $\|u_0\|_{\dot B^{-2/p}_{p,q}}\le \eta$,
the system \eqref{eq:NS} has a mild solution in $L^p_t(L^q_x)$.
\end{theorem}

\begin{proof}
The proof follows the lines of the existence result for \eqref{NLHE1} using Theorem~\ref{thm:fixedpoint}, having in mind that ${\mathbb{P}}$ 
is a bounded operator on every $W^{s,q}_x$ if $s\in{\mathds{R}}$ and $1<q<\infty$.
\end{proof}

In the limit $p=\infty$, the relation $\frac{2}{p}+\frac{n}{q}=1$ imposes $q=n$. We state next a uniqueness result for mild solutions
in ${\mathscr{C}}([0,T];L^n_x)$. This result was originally proved by Furioli, Lemari\'e-Rieusset and Terraneo in \cite{FLRT00} by different
methods and in \cite{Mo99} using maximal regularity (see also the book by Lemari\'e-Rieusset \cite{LR02} where he gave at least 3 proofs in 
Chapter~27). For the case of Navier-Stokes equations in bounded Lipschitz domains, see for instance \cite{Mo02} .

\begin{theorem}[Uniqueness]
Let $T>0$. For all $u_0\in L^q_x$ there exists at most one mild solution
in ${\mathscr{C}}([0,T];L^n_x)$ of \eqref{eq:NS}.
\end{theorem}

\begin{proof}
Again, the strategy of the proof and the tools used are the same as in the case of \eqref{NLHE1}. 
We assume that there are two mild solutions $u$ and $v$ of \eqref{eq:NS}, i.e., $u$ and $v$ satisfy both
\eqref{def:mildsolNS}. We prove that for some (all) $p\in(1,\infty)$, there exists $\tau\in(0,T]$ such that
 $\|u-v\|_{L^p(0,\tau;L^n_x)}\le c \|u-v\|_{L^p(0,\tau;L^n_x)}$ with $0<c<1$, so that $u=v$ on $[0,\tau]$ (by
 continuity) and then iterate the procedure.
\end{proof}

\section*{Acknowledgements}
The research is financially supported by a grant ANR-18-CE40-0012 ``Real Analysis and GEometry" (french
research agency).


\begin{thebibliography}{99}

\bibitem{BCP62}
A. Benedek, A.-P. Calder\'{o}n and P. Panzone,
Convolution operators on Banach space valued functions,
Proc. Nat. Acad. Sci. U.S.A. {\bf 48} (1962), 356--365.

\bibitem{CD00}
T. Coulhon and X. T. Duong,
Maximal regularity and kernel bounds: observations on a theorem by Hieber and Pr\"{u}ss,
Adv. Differential Equations {\bf 5} n$^{\text o}$1-3 (2000), 343--368.

\bibitem{CL86}
T. Coulhon and D. Lamberton,
R\'{e}gularit\'{e} $L^p$ pour les \'{e}quations d'\'{e}volution,
S\'{e}minaire d'{A}nalyse {F}onctionelle 1984/1985, Publ. Math. Univ. Paris VII {\bf 26} (1986), 155--165.

\bibitem{DV87}
G. Dore and A. Venni,
On the closedness of the sum of two closed operators,
Math. Z. {\bf 196} n$^{\text o}$2 (1987), 189--201.

\bibitem{FK64}
H. Fujita and T. Kato,
On the Navier-Stokes initial value problem. I,
Arch. Rational Mech. Anal. {\bf 16} (1964), 269--315.

\bibitem{FLRT00}
G. Furioli, P. G. Lemari\'{e}-Rieusset and E. Terraneo, 
Unicit\'{e} dans $L^3({\mathds{R}}^3)$ et d'autres espaces fonctionnels limites pour Navier-Stokes,
Rev. Mat. Iberoamericana {\bf 16} n$^{\text o}$3 (2000), 605--667.

\bibitem{HP97}
M. Hieber and J. Pr{\"u}ss,
Heat kernels and maximal $L^p$-$L^q$ estimates for parabolic evolution equations,
Comm. Partial Differential Equations {\bf 22} n$^{\text o}$9-10 (1997), 1647--1669.

\bibitem{Ka84}
T. Kato,   
Strong $L^{p}$-solutions of the Navier-Stokes equation in ${\mathds{R}}^{m}$, with applications to weak solutions,
Math. Z. {\bf 187} n$^{\text o}$4 (1984), 471--480.

\bibitem{Ku08}
P. C. Kunstmann, 
On maximal regularity of type $L^p\text{-}L^q$ under minimal assumptions for elliptic non-divergence operators,
J. Funct. Anal. {\bf 255} n$^{\text o}$10 (2008), 2732--2759.

\bibitem{La87}
D. Lamberton, 
\'{E}quations d'\'{e}volution lin\'{e}aires associ\'{e}es \`a des semi-groupes de contractions dans les espaces $L^p$,
J. Funct. Anal. {\bf 72} n$^{\text o}$2 (1987), 252--262.

\bibitem{LR02}
P. G. Lemari\'{e}-Rieusset, 
Recent developments in the Navier-Stokes problem,
Chapman \& Hall/CRC Research Notes in Mathematics {\bf 431},
Chapman \& Hall/CRC, Boca Raton, FL (2002), xiv+395 pages.

\bibitem{Mo99}
S. Monniaux, 
Uniqueness of mild solutions of the Navier-Stokes equation and maximal $L^p$-regularity,
C. R. Acad. Sci. Paris S\'{e}r. I Math. {\bf 328} n$^{\text o}$8 (1999), 663--668.

\bibitem{Mo02}
S. Monniaux,   
On uniqueness for the Navier-Stokes system in 3D-bounded Lipschitz domains,
J. Funct. Anal. {\bf 195} n$^{\text o}$1 (2002), 1--11.

\bibitem{PS04}
J. Pr\"{u}ss and G. Simonett, 
Maximal regularity for evolution equations in weighted $L_p$-spaces,
Arch. Math. (Basel) {\bf 82} n$^{\text o}$5 (2004), 415--431.

\bibitem{DeS64}
L. de Simon, 
Un'applicazione della teoria degli integrali singolari allo studio
delle equazioni differenziali lineari astratte del primo ordine, 
Rend. Sem. Mat. Univ. Padova {\bf 34} (1964), 205--223.

\bibitem{We01}
L. Weis,
Operator-valued Fourier multiplier theorems and maximal $L_p$-regularity,
Math. Ann. {\bf 319} n$^{\text o}$4 (2001),735--758.

\bibitem{W80}
F. B. Weissler, 
Local existence and nonexistence for semilinear parabolic equations in $L^{p}$,
Indiana Univ. Math. J. {\bf 29} n$^{\text o}$1 (1980), 79--102.
 
\bibitem{W81}
F. B. Weissler, 
Existence and nonexistence of global solutions for a semilinear heat equation,
Israel J. Math. {\bf 38} n$^{\text o}$1-2 (1981), 29--40.
 
\end{thebibliography}
\end{document}